\documentclass[11pt,twoside]{article}
\usepackage{mathrsfs}
\usepackage{amsmath}
\usepackage{amsthm}
\input amssymb.sty
\usepackage{amsfonts}
\usepackage{amssymb}
\usepackage{latexsym}
\usepackage[all]{xy}

\date{\empty}
\allowdisplaybreaks[4] \footskip=15pt

\numberwithin{equation}{section} \theoremstyle{plain}
\newtheorem*{thm*}{Main Theorem}
\newtheorem{theorem}{Theorem}[section]
\newtheorem{corollary}[theorem]{Corollary}
\newtheorem*{corollary*}{Corollary}

\newtheorem*{claim*}{Claim}
\newtheorem{lemma}[theorem]{Lemma}
\newtheorem*{lemma*}{Lemma}
\newtheorem{proposition}[theorem]{Proposition}
\newtheorem*{proposition*}{Proposition}
\newtheorem{remark}[theorem]{Remark}
\newtheorem*{remark*}{Remark}
\newtheorem{example}[theorem]{Example}
\newtheorem*{example*}{Example}

\newtheorem*{question*}{Question}

\newtheorem*{definition*}{Definition}

\newtheorem*{acknowledgements*}{ACKNOWLEDGEMENTS}

\begin{document}
\begin{center}
{\large  \bf Characterizations and representations of core and dual core inverses}\\
\vspace{0.8cm} {\small \bf   Jianlong
Chen$^{[1]}$\footnote{Corresponding author

1 Department of Mathematics, Southeast University, Nanjing 210096, China.

2 CMAT-Centro de Matem\'{a}tica, Universidade do Minho, Braga
4710-057, Portugal.

3 Departamento de Matem\'{a}tica e Aplica\c{c}\~{o}es, Universidade
do Minho, Braga 4710-057, Portugal.

Email: jlchen@seu.edu.cn(J. Chen), ahzhh08@sina.com(H. Zhu),
pedro@math.uminho.pt (P. Patr\'{i}cio), Zhang@math.uminho.pt(Y.
Zhang).}, Huihui Zhu$^{[1,2]}$, Pedro Patr\'{i}cio$^{[2,3]}$, Yulin
Zhang$^{[2,3]}$}
\end{center}

\bigskip

{ \bf  Abstract:}  \leftskip0truemm\rightskip0truemm  In this paper,
double commutativity and the reverse order law for the core inverse
are considered. Then, new characterizations of the Moore-Penrose
inverse of a regular element are given by one-sided invertibilities
in a ring. Furthermore, the characterizations and representations of
the core and dual core inverses of a regular element are considered.


\bigskip



\section { \bf Introduction}

Let $R$ be an associative ring with unity 1. We say that $a\in R$ is
(von Neumann) regular if there exists $x\in R$ such that $axa=a$.
Such $x$ is called an inner inverse of $a$, and is denoted by
$a^{-}$. Let $a\{1\}$ be the set of all inner inverses of $a$.
Recall that an element $a\in R$ is said to be group invertible if
there exists $x\in R$ such that $axa=a$, $xax=x$ and $ax=xa$. The
element $x$ satisfying the conditions above is called a group
inverse of $a$. The group inverse of $a$ is unique if it exists, and
is denoted by $a^\#$.

Throughout this paper, assume that $R$ is a unital $*$-ring, that is
a ring with unity 1 and an involution $a \mapsto a^*$ such that
$(a^*)^* = a$, $(a+b)^* = a^* + b^*$ and $(ab)^* = b^*a^*$ for all
$a,b\in R$. An element $a\in R$ is called Moore-Penrose invertible
\cite{Penrose} if there exists $x\in R$ satisfying the following
equations
 \begin{center}
 ${\rm(i)}~axa=a$,~~ ${\rm (ii)}~ xax=x$,~~ ${\rm (iii)}~(ax)^*=ax$,~~ ${\rm (iv)}~(xa)^*=xa$.
\end{center}
Any element $x$ satisfying the equations (i)-(iv) is called a
Moore-Penrose inverse of $a$. If such $x$ exists, it is unique and
is denoted by $a^\dag$. If $x$ satisfies the conditions (i) and
(iii), then $x$ is called a $\{1,3\}$-inverse of $a$, and is denoted
by $a^{(1,3)}$. If $x$ satisfies the conditions (i) and (iv), then
$x$ is called a $\{1,4\}$-inverse of $a$, and is denoted by
$a^{(1,4)}$. The symbols $a$\{1,3\} and $a$\{1,4\} denote the sets
of all \{1,3\}-inverses and \{1,4\}-inverses of $a$, respectively.

The concept of core inverse of a complex matrix was first introduced
by Baksalary and Trenkler \cite{Baksalary and Trenkler}. Recently,
Raki\'{c} et al. \cite{Serbia} generalized the definition of core
inverse to the ring case. An element $a\in R$ is core invertible
(see \cite[Definition 2.3]{Serbia}) if there exists $x\in R$ such
that $axa=a$, $xR=aR$ and $Rx=Ra^*$. It is known that the core
inverse $x$ of $a$ is unique if it exists, and is denoted by
$a^{(\#)}$. The dual core inverse of $a$ when exists is defined as
the unique $a_{(\#)}$ such that $aa_{(\#)}a=a$, $a_{(\#)}R =a^* R$
and $Ra_{(\#)}=Ra$. By  $R^{-1}$, $R^\#$, $R^\dag$, $R^{(1,3)}$,
$R^{(1,4)}$, $R^{(\#)}$ and $R_{(\#)}$ we denote the sets of all
invertible, group invertible, Moore-Penrose invertible,
$\{1,3\}$-invertible, $\{1,4\}$-invertible, core invertible and dual
core invertible elements in $R$, respectively.

In this paper, double commutativity and the reverse order law for
the core inverse proposed in \cite{Baksalary} are considered. Also,
we characterize the Moore-Penrose inverse of a regular element by
one-sided invertibilities in a ring $R$. Furthermore, new existence
criteria and representations of core inverse and dual core inverse
of a regular element are given by units.

\section{Some lemmas}

The following lemmas will be useful in the sequel.

\begin{lemma} \label{Jlemma} Let $a,b\in R$. Then

\emph{(i)} If there exists $x\in R$ such that $(1+ab)x=1$, then
$(1+ba)(1-bxa)=1$.

\emph{(ii)} If there exists $y\in R$ such that $y(1+ab)=1$, then
$(1-bya)(1+ba)=1$.
\end{lemma}

According to Lemma \ref{Jlemma}, we know that $1+ab \in R^{-1}$ if
and only if $1+ba\in R^{-1}$. In this case,
$(1+ba)^{-1}=1-b(1+ab)^{-1}a$, which is known as Jacobson's Lemma.

\begin{lemma} \label{1314} {\rm \cite[p. 201]{Hartwig}} Let $a,x\in R$.  Then

\emph{(i)}  $x$ is a $\{1,3\}$-inverse of $a$ if and only if
$a^*=a^*ax$.

\emph{(ii)}   $x$ is a $\{1,4\}$-inverse of $a$ if and only if
$a=aa^*x^*$.
\end{lemma}

It is known that $a\in R^\dag$ if and only if $a\in aa^*R \cap
Ra^*a$ if and only if $a\in R^{(1,3)} \cap R^{(1,4)}$. In this case,
$a^\dag=a^{(1,4)}aa^{(1,3)}$. By Lemma \ref{1314}, we know that
$a=xa^*a=aa^*y$ implies $a\in R^\dag$ and $a^\dag=y^*ax^*$.

\begin{lemma} \label{star regular}{ \rm \cite[Theorems 2.16, 2.19 and 2.20]{Zhu and chen}} Let $S$ be a $*$-semigroup and let $a\in S$.
Then $a$ is Moore-Penrose invertible if and only if $a\in aa^*aS$ if
and only if $a\in Saa^*a$. Moreover, if $a=aa^*ax=yaa^*a$ for some
$x,y\in S$, then $a^\dag=a^*ax^2a^*=a^*y^2aa^*$.
 \end{lemma}

\begin{lemma} \label{Hartwig group inverse} {\rm \cite[Proposition 7]{Hartwig}} Let $a\in R$. Then $a\in R^\#$ if and only
if $a=a^2x=ya^2$ for some $x, y\in R$. In this case,
$a^\#=yax=y^2a=ax^2$.
\end{lemma}

\begin{lemma} \label{expression} {\rm \cite[Theorems 2.6 and 2.8]{Xu}} Let $a\in R$. Then

\emph{(i)} $a\in R^{(\#)}$ if and only if $a\in R^\# \cap
R^{(1,3)}$. In this case, $a^{(\#)}=a^\#aa^{(1,3)}$.

\emph{(ii)} $a\in R_{(\#)}$ if and only if $a\in R^\# \cap
R^{(1,4)}$. In this case, $a_{(\#)}=a^{(1,4)}aa^\#$.
\end{lemma}

\begin{lemma} {\rm \cite[Theorem 2.14]{Serbia} and \cite[Theorem 3.1]{Xu}} \label{core 5} Let $a\in R$. Then
 $a\in R^{(\#)}$ with core inverse $x$
if and only if $axa=a$, $xax=x$, $(ax)^*=ax$, $xa^2=a$ and $ax^2=x$
if and only if $(ax)^*=ax$, $xa^2=a$ and $ax^2=x$.
\end{lemma}

\section{Double commutativity and reverse order law for core inverses}

First, we give the following lemma to prove the double commutativity
of core inverse.

\begin{lemma} \label{commute} Let $a,b,x\in R$ with $xa=bx$ and $xa^*=b^*x$. If $a,b\in R^{(1,3)}$, then
\begin{center}
$xaa^{(1,3)} = bb^{(1,3)}x$.
\end{center}
\end{lemma}

\begin{proof} From $xa=bx$, it follows that
\begin{eqnarray*}
  xaa^{(1,3)} &=&bxa^{(1,3)}=bb^{(1,3)}bxa^{(1,3)} \\
   &=& bb^{(1,3)}xaa^{(1,3)}.
\end{eqnarray*}

The condition $xa^*=b^*x$ implies that
\begin{eqnarray*}
  bb^{(1,3)}x &=& (b^{(1,3)})^*b^*x=(b^{(1,3)})^*xa^* \\
   &=& (b^{(1,3)})^*x(aa^{(1,3)}a)^*\\
   &=&(b^{(1,3)})^*xa^*aa^{(1,3)}\\
   &=&(b^{(1,3)})^*b^*xaa^{(1,3)}\\
   &=& bb^{(1,3)}xaa^{(1,3)}.
\end{eqnarray*}

Hence, $xaa^{(1,3)} = bb^{(1,3)}x$.
\end{proof}

\begin{theorem} \label{double commute} Let $a,b,x\in R$ with $xa=bx$ and $xa^*=b^*x$. If $a,b \in R^{(\#)}$, then
$xa^{(\#)}=b^{(\#)}x$.
\end{theorem}

\begin{proof} As $a,b\in R^{(\#)}$, then  $a,b\in R^\#$ from Lemma \ref{expression}. Applying \cite[Theorem 2.2]{Drazin}, we
get $b^\#x=xa^\#$ since $xa=bx$.

So, $xa^{(\#)}=b^{(\#)}x$. Indeed,
$xa^{(\#)}=xa^\#aa^{(1,3)}=b^\#xaa^{(1,3)}=b^\#bb^{(1,3)}x=b^{(\#)}x$.
\end{proof}

\begin{remark} {\rm Theorem \ref{double commute} above can also been obtained from
\cite[Theorem 2.3]{Drazin}. Indeed, note in \cite[Theorem
4.4]{Serbia} that $a$ has ($a$, $a^*$)-inverse if and only if $a\in
R^{(\#)}$.}
\end{remark}

\begin{corollary} Let $a,x\in R$ with $xa=ax$ and $xa^*=a^*x$. If $a\in R^{(\#)}$, then
$xa^{(\#)}=a^{(\#)}x$.
\end{corollary}

In 2012, Baksalary and Trenkler \cite{Baksalary} asked the following
question: Given complex matrices $A$ and $B$, if $A^{(\#)}$,
$B^{(\#)}$ and $(AB)^{(\#)}$ exist, does it follow that
$(AB)^{(\#)}= B^{(\#)}A^{(\#)}$? Later, Cohen, Herman and Jayaraman
\cite{Cohen Herman} presented several counterexamples for this
problem.

Next, we show that the reverse order law for the core inverse holds
under certain conditions in a general ring case.

\begin{theorem} \label{reverse order} Let $a,b\in R^{(\#)}$ with $ab=ba$ and $ab^*=b^*a$. Then
$ab \in R^{(\#)}$ and
$(ab)^{(\#)}=b^{(\#)}a^{(\#)}=a^{(\#)}b^{(\#)}$.
\end{theorem}

\begin{proof} It follows from Theorem \ref{double commute} that $b^{(\#)}a=ab^{(\#)}$ and
$a^{(\#)}b=ba^{(\#)}$.

Also, the conditions $b^*a=ab^*$ and $a^*b^*=b^*a^*$ guarantee that
$b^*a^{(\#)}=a^{(\#)}b^*$, which together with $a^{(\#)}b=ba^{(\#)}$
imply $a^{(\#)}b^{(\#)}=b^{(\#)}a^{(\#)}$ according to Theorem
\ref{double commute}.

Once given the above conditions, it is straightforward to check

(1) By Lemma \ref{commute}, we have $abb^{(1,3)}=bb^{(1,3)}a$.
Hence,
$abb^{(\#)}a^{(\#)}ab=abb^{(1,3)}aa^\#b=bb^{(1,3)}aaa^\#b=bb^{(1,3)}ba=ab$.

(2) Since $abb^{(1,3)}=bb^{(1,3)}a$, it follows that
\begin{eqnarray*}
b^{(\#)}a^{(\#)}&=&b^\#bb^{(1,3)}a^\#aa^{(1,3)}=b^\#bb^{(1,3)}aa^\#a^{(1,3)}\\
&=&b^\#abb^{(1,3)}a^\#a^{(1,3)}= ab^\#bb^{(1,3)}a^\#a^{(1,3)}\\
&=&a bb^\#b^{(1,3)}a^\#a^{(1,3)}
\end{eqnarray*}
and
\begin{eqnarray*}
ab&=&b^\#b^2a=b^\#bb^{(1,3)}b^2a\\
&=&b^{(\#)}ab^2=b^{(\#)}a^\#aa^{(1,3)}a^2b^2\\
&=&b^{(\#)}a^{(\#)}a^2b^2.
\end{eqnarray*}

Hence, $abR=b^{(\#)}a^{(\#)}R$.

(3) If $x$ in Lemma \ref{commute} is group invertible, then
$aa^{(1,3)} x^\#= x^\#aa^{(1,3)}$. We have

\begin{eqnarray*}
b^{(\#)}a^{(\#)}&=&b^\#bb^{(1,3)}a^\#aa^{(1,3)}=b^\#a^\#bb^{(1,3)}aa^{(1,3)}\\
&=&b^\#a^\#(aa^{(1,3)}bb^{(1,3)})^*=b^\#a^\#(baa^{(1,3)}b^{(1,3)})^*\\
&=&b^\#a^\#(a^{(1,3)}b^{(1,3)})^*(ab)^*
\end{eqnarray*}
and
\begin{eqnarray*}
(ab)^*&=&b^*a^*aa^{(1,3)}=a^* b^*a
a^{(1,3)}=a^*b^*bb^{(1,3)}aa^{(1,3)}\\
&=&b^*a^*aa^\#abb^{(1,3)}a^{(1,3)}=b^*a^*abb^{(1,3)}a^\#aa^{(1,3)}\\
&=&b^*a^*abb^\#bb^{(1,3)} a^\#aa^{(1,3)}\\
&=&b^*a^*ab b^{(\#)}a^{(\#)}.
\end{eqnarray*}

Thus, $Rb^{(\#)} a^{(\#)}=R(ab)^*$.

So, $ab\in R^{(\#)}$ and
$(ab)^{(\#)}=b^{(\#)}a^{(\#)}=a^{(\#)}b^{(\#)}$.
\end{proof}

\section{Characterizations of core inverses by units}

In this section, we give existence criteria for the core inverse of
ring elements in terms of units. Representations based on classical
inverses are also given. By duality, all the results apply to the
dual core inverse.

We now present an existence criterion of group inverse of a regular
element.

\begin{proposition} \label{group inverse}
Let $k\geq 1$ be an integer and suppose that $a\in R$ is regular
with an inner inverse $a^{-}$. Then the following conditions are
equivalent{\rm :}

\emph{(i)} $a\in R^\#$.

\emph{(ii)} $u=a^k+1-aa^{-}\in R^{-1}$.

\emph{(iii)} $v=a^k+1-a^{-}a\in R^{-1}$.

In this case, $a^\#=u^{-1}a^{2k-1}v^{-1}$.

\end{proposition}
\begin{proof}
(i) $\Rightarrow$ (ii). Since
 \begin{eqnarray*}
 u(a(a^\#)^ka^{-}+1-aa^\#)&=& (a^k+1-aa^{-})(a(a^\#)^ka^{-}+1-aa^\#) \\
 &=& a^{k+1}(a^\#)^ka^{-}+1-aa^{-}\\
 &=& aa^{-}+1-aa^{-}\\
 &=&1,
 \end{eqnarray*}
it follows that $u$ is right invertible.

Similarly, we can prove $(a(a^\#)^ka^{-}+1-aa^\#)u=1$, i.e., $u$ is
left invertible.

Hence, $u=a^k+1-aa^{-}\in R^{-1}$.

(ii) $\Leftrightarrow$ (iii). Note that $u=1+a(a^{k-1}-a^{-})\in
R^{-1}$ if and only if $1+(a^{k-1}-a^{-})a=v\in R^{-1}$.

(iii) $\Rightarrow$ (i). As $v \in R^{-1}$, then $u\in R^{-1}$.
Since $ua=a^{k+1}=av$, it follows that $a=a^{k+1}v^{-1} =
u^{-1}a^{k+1} \in a^2R \cap Ra^2$, i.e., $a\in R^\#$.

Note that $a=u^{-1}a^{k-1}a^2=a^2a^{k-1}v^{-1}\in a^2R \cap Ra^2$.
It follows from Lemma \ref{Hartwig group inverse} that
$a^\#=u^{-1}a^{k-1}aa^{k-1}v^{-1}=u^{-1}a^{2k-1}v^{-1}$.
\end{proof}

\begin{theorem} \label{core inverse} Let $a\in R$ be regular. Then the following conditions are equivalent{\rm :}

\emph{(i)} $a\in R^{(\#)}$.

\emph{(ii)} $a+1-aa^{-}$ and $a^*+1-aa^{-}$ are invertible for some
$a^{-}\in a\{1\}$.

\emph{(iii)} $a+1-aa^{-}$ is invertible and $a^*+1-aa^{-}$ is left
invertible for some $a^{-}\in a\{1\}$.

\emph{(iv)} $a^*a+1-aa^{-}$ and $(a^*)^2+1-aa^{-}$ are invertible
for some $a^{-}\in a\{1\}$.

\emph{(v)} $a^*a+1-aa^{-}$ and $(a^*)^2+1-aa^{-}$ are left
invertible for some $a^{-}\in a\{1\}$.

\emph{(vi)} $a+1-aa^{-}$ and $(a^*)^2+1-aa^{-}$ are left invertible
for some $a^{-}\in a\{1\}$.

In this case,
\begin{eqnarray*}
a^{(\#)}&=&(a^*a+1-aa^{-})^{-1}a^*=a[((a^*)^2+1-aa^{-})^{-1}]^*\\
&=&(a+1-aa^{-})^{-1}a((a^*+1-aa^{-})^{-1})^*.
\end{eqnarray*}
\end{theorem}

\begin{proof} (i) $\Rightarrow$ (ii). Since $a\in R^{(\#)}$ then $a\in R^\# \cap R^{(1,3)}$ by Lemma \ref{expression}.
Let $a^{-} \in a\{1,3\}$. Then $a+1-aa^{-}$ is invertible by
Proposition \ref{group inverse} and hence
$a^*+1-aa^{-}=(a+1-aa^{-})^*$ is invertible.

(ii) $\Rightarrow$ (iii) is clear.

(iii) $\Rightarrow$ (i). As $a^*+1-aa^{-}$ is left invertible, then
there exists $s\in R$ such that $s(a^*+1-aa^{-})=1$. Hence,
$a=s(a^*+1-aa^{-})a=sa^*a\in Ra^*a$, i.e., $a^{(1,3)}$ exists by
Lemma \ref{1314}(i). Also, $a+1-aa^{-}\in R^{-1}$ implies that
$a^\#$ exists by Proposition \ref{group inverse}. So, $a\in
R^{(\#)}$ by Lemma \ref{expression}.

(i) $\Rightarrow$ (iv). Let $a^{-} \in a\{1,3\}$. Then $a+1-aa^{-}$
and $a^*+1-aa^{-}$ are invertible. Hence,
$a^*a+1-aa^{-}=(a^*+1-aa^{-})(a+1-aa^{-})$ is invertible.

Also, it follows from Proposition \ref{group inverse} that
$a^2+1-aa^{-}\in R^{-1}$ since $a\in R^\#$. So,
$(a^*)^2+1-aa^{-}=(a^2+1-aa^{-})^*\in R^{-1}$.

(iv) $\Rightarrow$ (v). Clearly.

(v) $\Rightarrow$ (i). Since $a^*a+1-aa^{-}$ and $(a^*)^2+1-aa^{-}$
are both left invertible, there exist $m,n\in R$ such that
$m(a^*a+1-aa^{-})=1=n((a^*)^2+1-aa^{-})$. As
$a=m(a^*a+1-aa^{-})a=ma^*a^2$ and
$a=n((a^*)^2+1-aa^{-})a=n(a^*)^2a$, then
$ma^*=m(n(a^*)^2a)^*=(ma^*a^2)n^*=an^*$.

Let $x=ma^*=an^*$. Then $a=(na^*)a^*a=x^*a^*a$ and hence $x$ is a
\{1,3\}-inverse of $a$ by Lemma \ref{1314}. So, we have $axa=a$ and
$(ax)^*=ax$. Also, $xa^2=ma^*a^2=a$ and
$ax^2=ax(an^*)=(axa)n^*=an^*=x$. It follows from Lemma \ref{core 5}
that $a\in R^{(\#)}$ and $a^{(\#)}=ma^*=an^*$.

(i) $\Rightarrow$ (vi) by (i) $\Rightarrow$ (iv) and Proposition
\ref{group inverse}.

(vi) $\Rightarrow$ (i). Let $u=a+1-aa^{-}$ and $v=(a^*)^2+1-aa^{-}$.
As $u$ and $v$ are left invertible, then there exist $s,t\in R$ such
that $su=tv=1$. Hence, $a=tva=t(a^*)^2a\in Ra^*a$, which implies
that $a\in R^{(1,3)}$ according to Lemma \ref{1314}(i). Also,
$a=t(a^*)^2a=ta^*(t(a^*)^2a)^*a=(t(a^*)^2a)at^*a=a^2t^*a\in a^2R$,
which combines with $a=sua=sa^2\in Ra^2$ conclude $a\in a^2R \cap
Ra^2$, i.e., $a\in R^\#$. So, $a\in R^{(\#)}$ by Lemma
\ref{expression}.

We next give another formulae of $a^{(\#)}$.

Note that (iv) $\Leftrightarrow$ (v). In the proof of (v)
$\Rightarrow$ (i), taking $m=(a^*a+1-aa^{-})^{-1}$ and
$n=((a^*)^2+1-aa^{-})^{-1}$.

We obtain
\begin{eqnarray*}
a^{(\#)}&=&ma^*=(a^*a+1-aa^{-})^{-1}a^*\\
&=&an^*=a[((a^*)^2+1-aa^{-})^{-1}]^*.
\end{eqnarray*}

As $(a+1-aa^{-})a=a^2$, then $a=(a+1-aa^{-})^{-1}a^2$ and hence
$a^\#=(a+1-aa^{-})^{-2}a$ by Lemma \ref{Hartwig group inverse}.

From $(a^*+1-aa^{-})a=a^*a$, it follows that
$a=(a^*+1-aa^{-})^{-1}a^*a$. Using Lemma \ref{1314}(i), we know that
$((a^*+1-aa^{-})^{-1})^*$ is a \{1,3\}-inverse of $a$.

So,
\begin{eqnarray*}
a^{(\#)}&=&a^\#aa^{(1,3)}\\
&=&(a+1-aa^{-})^{-2}a^2((a^*+1-aa^{-})^{-1})^*\\
&=&(a+1-aa^{-})^{-1}a((a^*+1-aa^{-})^{-1})^*.
\end{eqnarray*}

The proof is completed.
\end{proof}

\begin{remark}
{\rm If $a\in R$ satisfies $a^*a=1$ and $aa^* \neq 1$, then
$a^*+1-aa^{-}$ is not left invertible for any $a^{-}\in a\{1\}$. In
fact, if $a^*+1-aa^{-}$ is left invertible for some $a^{-}\in
a\{1\}$, then there exists $s$ such that $s(a^*+1-aa^{-})=1$. As
$a^*a=1$, then $a=s(a^*+1-aa^{-})a=sa^*a=s$. Hence,
$a(a^*+1-aa^{-})=1$ and $a\in R^{-1}$. So, $aa^*=1$, which is a
contradiction.}
\end{remark}

\begin{proposition} \label{one} Let $k\geq 1$ be an integer and suppose that $a\in R$ is regular. If $(a^*)^k+1-aa^{-}\in R^{-1}$ for any $a^{-}\in a\{1\}$,
then $a\in R^{(\#)}$.
\end{proposition}

\begin{proof} Let $u=(a^*)^k+1-aa^{-}$. As $u$ is invertible, then $a=u^{-1}(a^*)^ka \in Ra^*a$, hence $a$ is $\{1,3\}$-invertible by Lemma \ref{1314}(i).

As $((a^*)^k+1-aa^{(1,3)})^*=a^k+1-aa^{(1,3)}$ is invertible for
$a^{(1,3)}\in a\{1\}$, then $a\in R^\#$ by Proposition \ref{group
inverse}. So, $a\in R^{(\#)}$ from Lemma \ref{expression}.
\end{proof}

\begin{remark} \label{rmk change}{\rm If $a^*+1-aa^{-}\in R^{-1}$ for some $a^{-}\in
a\{1\}$, then $a\notin R^{(\#)}$ in general. Such as let
$R=M_2(\mathbb{C})$ be the ring of all $2 \times 2$ complex matrices
and suppose that involution $*$ is the conjugate transpose. Let
$A=\left[\begin{smallmatrix}
              0   & 1 \\
              0 & 0
        \end{smallmatrix}
 \right] \normalsize
\in R$. Then $A^{-}=\left[\begin{smallmatrix}
              0   & 0 \\
              1 & 1
        \end{smallmatrix}
 \right] \normalsize \in A\{1\}$. Hence, $A^*+I-AA^{-}=
 \left[\begin{smallmatrix}
              0 & -1 \\
              1 & 1
        \end{smallmatrix}
 \right]  \normalsize  \in R^{-1}$, but $A\notin R^\#$. So, $A\notin R^{(\#)}$.}
\end{remark}

The converse of Proposition \ref{one} may not be true. In the
following Example \ref{ex}, we find $a$ core invertible, but there
exists some $a^{-} \in a\{1\}$ such that none of $a^*+1-aa^{-}$,
$(a^*)^2+1-aa^{-}$ and $a^*a+1-aa^{-}$ are invertible.

\begin{example} \label{ex} {\rm Let $R$ be the ring as Remark \ref{rmk change}.
Given $A=\left[\begin{smallmatrix}
              1 & -2 \\
              1 & -2 \\
        \end{smallmatrix}
 \right] \normalsize \in R$, then $A^2=-A$ and hence $A^\#$ exists. So,
$A^{(\#)}$ exists. Taking $A^{-}=\left[\begin{smallmatrix}
            \frac{2}{3} & \frac{1}{3} \\
            0 & 0 \\
        \end{smallmatrix}
 \right]$, then
$A^*+I-AA^{-}=\frac{1}{3} \left[\begin{smallmatrix}
             4 & 2 \\
             -8 & -4 \\
        \end{smallmatrix}
 \right]$,
 $(A^*)^2+I-AA^{-}= \frac{1}{3}
 \left[\begin{smallmatrix}
              -2 & -4 \\
  4 & 8 \\
        \end{smallmatrix}
 \right]$ and
  $A^*A+I-AA^{-}=  \frac{1}{3}
  \left[\begin{smallmatrix}
               7 & -13 \\
   -14 & 26 \\
        \end{smallmatrix}
 \right]$ are not invertible.}
\end{example}

\begin{theorem} Let  $k\geq 1$ be an integer and suppose $a\in R^{(\#)}$. Then the
following conditions are equivalent for any $a^{-}\in a\{1\}${\rm :}

\emph{(i)} $(a^*)^k+1-aa^{-}\in R^{-1}$.

\emph{(ii)} $(a^*)^{k+1}+1-aa^{-}\in R^{-1}$.

\emph{(iii)} $a^*a+1-aa^{-}\in R^{-1}$.

In this case,
$a^{(\#)}=(a^*a+1-aa^{-})^{-1}a^*=a^k[((a^*)^{k+1}+1-aa^{-})^{-1}]^*$.
\end{theorem}

\begin{proof} As $a\in R^{(\#)}$, then $a\in
R^\#$ by Lemma \ref{expression}. Hence, $a+1-aa^{(\#)}\in R^{-1}$
from Proposition \ref{group inverse}. Note that
$aa^{(\#)}=aa^{(1,3)}$ and $a^*aa^{(\#)}=a^*$. Hence,
$1+(a^*-1)aa^{(\#)}=a^*+1-aa^{(\#)}=(a+1-aa^{(\#)})^*\in R^{-1}$.
From Jacobson's Lemma, it follows that
$aa^{(\#)}a^*+1-aa^{(\#)}=1+aa^{(\#)}(a^*-1)\in R^{-1}$.

As $(a^*+1-aa^{-})(a+1-aa^{(\#)})=a^*a+1-aa^{-}$ and
$a+1-aa^{(\#)}\in R^{-1}$, then $a^*+1-aa^{-}\in R^{-1}$ if and only
if $a^*a+1-aa^{-}\in R^{-1}$.

Also,
$((a^*)^n+1-aa^{-})(aa^{(\#)}a^*+1-aa^{(\#)})=(a^*)^naa^{(\#)}a^*+1-aa^{-}=(a^*)^{n+1}+1-aa^{-}$,
then $(a^*)^n+1-aa^{-}\in R^{-1}$ if and only if
$(a^*)^{n+1}+1-aa^{-}\in R^{-1}$ since $aa^{(\#)}a^*+1-aa^{(\#)}\in
R^{-1}$.

Thus, $a^*a+1-aa^{-}\in R^{-1}$ if and only if $(a^*)^k+1-aa^{-}\in
R^{-1}$ if and only if $(a^*)^{k+1}+1-aa^{-}\in R^{-1}$.

Set $m_1=(a^*a+1-aa^{-})^{-1}$ and $n_1=((a^*)^{k+1}+1-aa^{-})^{-1}
(a^*)^{k-1}$. Then
\begin{eqnarray*}
a&=&(a^*a+1-aa^{-})^{-1}(a^*a+1-aa^{-})a\\
 &=&m_1a^*a^2
\end{eqnarray*}
and
\begin{eqnarray*}
a &=& ((a^*)^{k+1}+1-aa^{-})^{-1}((a^*)^{k+1}+1-aa^{-})a\\
&=& ((a^*)^{k+1}+1-aa^{-})^{-1} (a^*)^{k-1}(a^*)^2a\\
&=&n_1(a^*)^2a.
\end{eqnarray*}

From Theorem \ref{core inverse} (v) $\Rightarrow$ (i), it follows
that
$a^{(\#)}=m_1a^*=an_1^*=(a^*a+1-aa^{-})^{-1}a^*=a^k[((a^*)^{k+1}+1-aa^{-})^{-1}]^*$.
\end{proof}

\begin{remark}{\rm Even though $a^*a+1-aa^{-}\in R^{-1}$ for any $a^{-}\in a\{1\}$, it does not imply the core invertibility of $a$. Let $R$ be the infinite
 matrix ring as in Remark \ref{ex1} and let $a=\Sigma_{i=1}^\infty e_{i+1,i}$. Then $a^*a=1$, $aa^* \neq 1$. For any  $a^{-}\in a\{1\}$,
 as $(2-aa^{-})^{-1}=\frac{1}{2}(1+aa^{-})$, then
 $2-aa^{-}=a^*a+1-aa^{-}\in R^{-1}$ . But $a \notin R^\#$ and hence
 $a\notin R^{(\#)}$.}
\end{remark}

\begin{proposition} \label{k core} Let $k\geq 1$ be an integer and suppose $a\in R$. Then the
following conditions are equivalent{\rm :}

\emph{(i)} $a\in R^{(\#)}$.

\emph{(ii)} $a\in R^{(1,3)}$ and $(a^*)^k+1-aa^{(1,3)}\in R^{-1}$
for any $a^{(1,3)}\in a\{1,3\}$.

\emph{(iii)} $a\in R^{(1,3)}$ and $(a^*)^k+1-aa^{(1,3)}\in R^{-1}$
for some $a^{(1,3)}\in a\{1,3\}$.

In this case,
$a^{(\#)}=(u^{-1})^*a^{2k-1}(u^{-1})^*=(u^{-1})^*a^{k-1}u^{-1}(a^k)^*$,
where $u=(a^*)^k+1-aa^{(1,3)}$.
\end{proposition}

\begin{proof} (i) $\Rightarrow$ (ii). It follows from Lemma
\ref{expression} that $a\in R^{(\#)}$ implies $a\in R^\# \cap
R^{(1,3)}$. Hence, $a^k+1-aa^{(1,3)}\in R^{-1}$ from Proposition
\ref{group inverse}.

So, $(a^*)^k+1-aa^{(1,3)}=(a^k+1-aa^{(1,3)})^*\in R^{-1}$ for any
$a^{(1,3)}\in a\{1,3\}$.

(ii) $\Rightarrow$ (iii) is clear.

(iii) $\Rightarrow$ (i). Let $u=(a^*)^k+1-aa^{(1,3)}$. Then
$a^k+1-aa^{(1,3)}=u^*\in R^{-1}$ and hence $a\in R^\#$ by
Proposition \ref{group inverse}.

As $u^*a=a^{k+1}$, then $a=(u^{-1})^*a^{k+1}=(u^{-1})^*a^{k-1}a^2$.
Lemma \ref{Hartwig group inverse} guarantees that
$a^\#=((u^{-1})^*a^{k-1})^2a$.

Also, $ua=(a^*)^ka$ implies $a=(u^{-1}(a^*)^{k-1})a^*a$. So,
applying Lemma \ref{1314}(i), we know that $a\in R^{(1,3)}$ and
$(u^{-1}(a^*)^{k-1})^*= a^{k-1} (u^{-1})^*$ is a \{1,3\}-inverse of
$a$.

Hence, we have
\begin{eqnarray*}
a^{(\#)}&=&a^\#aa^{(1,3)}\\
&=&((u^{-1})^*a^{k-1})^2a^2 a^{k-1} (u^{-1})^*\\
&=&(u^{-1})^*a^{k-1}(u^{-1})^*a^{k+1}a^{k-1} (u^{-1})^*\\
&=&(u^{-1})^*a^{k-1}aa^{k-1} (u^{-1})^*\\
&=&(u^{-1})^*a^{2k-1} (u^{-1})^*.
\end{eqnarray*}

From $ua^k=(a^*)^ku^*$, it follows that
$u^{-1}(a^*)^k=a^k(u^{-1})^*$.

Thus, $a^{(\#)}=(u^{-1})^*a^{k-1}
a^k(u^{-1})^*=(u^{-1})^*a^{k-1}u^{-1}(a^*)^k$.
\end{proof}

\begin{remark} {\rm In Proposition \ref{k core}, if $k\geq 2$, then the
expression of the core inverse of $a$ can be given as
$a^{(\#)}=a^{k-1}(u^{-1})^*$, where $u=(a^*)^k+1-aa^{(1,3)}$.
Indeed, as $u^*a^{k-1}=a^{2k-1}$, then $(u^*)^{-1}a^{2k-1}=a^{k-1}$.
Hence, $a^{(\#)}=(u^{-1})^*a^{2k-1}(u^{-1})^*=a^{k-1}(u^{-1})^*$.}
\end{remark}

Taking $k=1$ in Proposition \ref{k core}, it follows that

\begin{corollary} \label{1 core} Let $a\in R$. Then the
following conditions are equivalent{\rm :}

\emph{(i)} $a\in R^{(\#)}$.

\emph{(ii)} $a\in R^{(1,3)}$ and $a^*+1-aa^{(1,3)}\in R^{-1}$ for
any $a^{(1,3)}\in a\{1,3\}$.

\emph{(iii)} $a\in R^{(1,3)}$ and $a^*+1-aa^{(1,3)}\in R^{-1}$ for
some $a^{(1,3)}\in a\{1,3\}$.

In this case, $a^{(\#)}=(u^{-1})^*a(u^{-1})^*=(u^{-1})^*u^{-1}a^*$,
where $u=a^*+1-aa^{(1,3)}$.
\end{corollary}

\begin{remark} {\rm Let $a\in R$ be regular with an inner inverse $a^{-}$. If $u=a^*+1-aa^{-}\in R^{-1}$,
then $(u^{-1})^*a(u^{-1})^*=(u^{-1})^*u^{-1}a^*$. In fact, as
$ua=a^*a$, then $a=u^{-1}a^*a$ and hence $(u^{-1})^*\in a\{1,3\}$ by
Lemma \ref{1314}. Thus, $a(u^{-1})^*=(a(u^{-1})^*)^*=u^{-1}a^*$ and
$(u^{-1})^*a(u^{-1})^*=(u^{-1})^*u^{-1}a^*$. Moreover, if $a\in
R^{(\#)}$, then $a^{(\#)} \neq (u^{-1})^*u^{-1}a^*$ in general.
Indeed, take $A=\left[\begin{smallmatrix}
1 & -2 \\
 1 & -2 \\
 \end{smallmatrix}
 \right]$ in Remark \ref{rmk change}. Then $A^\#=A$ and $A^\dag= \frac{1}{10}
 \left[\begin{smallmatrix}
 1 & 1 \\
-2 & -2 \\
\end{smallmatrix}
 \right]$. Hence,
$A^{(\#)}=A^\#AA^\dag= -\frac{1}{2} \left[\begin{smallmatrix}
1 & 1 \\
1 & 1 \\
\end{smallmatrix}
\right]$. Taking $A^{-}=\left[\begin{smallmatrix}
0 & 1 \\
0 & 0 \\
\end{smallmatrix}
 \right]  \normalsize \in A\{1\}$, then $U=A^*+I-AA^{-}=\left[\begin{smallmatrix}
 2 & 0 \\
-2 & -2 \\
\end{smallmatrix}
 \right]$  \normalsize is invertible. But
$(U^{-1})^*U^{-1}A^*=(U^{-1})^*A(U^{-1})^*=\frac{1}{4}
\left[\begin{smallmatrix}
 0 & 0 \\
-1 & -1 \\
\end{smallmatrix}
 \right] \neq A^{(\#)}$.}
\end{remark}

\begin{proposition} Let $a\in R^{(\#)}$ and
suppose $u=a^*+1-aa^{-}\in R^{-1}$ for some $a^{-}\in a\{1\}$. Then
$a^{(\#)}=(u^{-1})^*a(u^{-1})^*$ if and only if $a^{-}\in a\{1,3\}$.
\end{proposition}

\begin{proof} ``$\Rightarrow$'' As $ua=a^*a$, then $a=u^{-1}a^*a$.
It follows from Lemma \ref{1314} that $(u^{-1})^*\in a\{1,3\}$ and
$a=a(u^{-1})^*a$. As also
$a=a^{(\#)}a^2=(u^{-1})^*a(u^{-1})^*a^2=(u^{-1})^*a^2=(u^*)^{-1}a^2$,
then $a^2=u^*a=(a+1-(aa^{-})^*)a=a^2+a-(aa^{-})^*a$, which implies
$a=(aa^{-})^*a=(a^{-})^*a^*a$. Again, Lemma \ref{1314} guarantees
that $a^{-}\in a\{1,3\}$.

``$\Leftarrow$'' See Corollary \ref{1 core}.
\end{proof}

Recall that a ring $R$ is called Dedekind-finite if $ab=1$ implies
$ba=1$, for all $a,b\in R$. We next give characterizations of core
inverse in such a ring.

\begin{corollary} \label{add new} Let $R$ be a Dedekind-finite ring. Then the following conditions are equivalent{\rm :}

\emph{(i)} $a\in R^{(\#)}$.

\emph{(ii)} $a\in R^{(1,3)}$ and $a^*a+1-aa^{(1,3)}$ is invertible
for any $a^{(1,3)}$.

\emph{(iii)} $a\in R^{(1,3)}$ and $a^*a+1-aa^{(1,3)}$ is invertible
for some $a^{(1,3)}$.

In this case, $a^{(\#)}=v^{-1}a^*$, where $v=a^*a+1-aa^{(1,3)}$.
\end{corollary}

\begin{proof} Let $u=a^*+1-aa^{(1,3)}$ and $v=a^*a+1-aa^{(1,3)}$.
Then $v=uu^*$. As $R$ is a Dedekind-finite ring, then $v\in R^{-1}$
if and only if $u\in R^{-1}$. By Corollary \ref{1 core},
$a^{(\#)}=(u^{-1})^*u^{-1}a^*=(uu^*)^{-1}a^*=v^{-1}a^*$.
\end{proof}

\section{Core, dual core and Moore-Penrose invertibility}

In this section, we mainly characterize the core inverse and dual
core inverse of ring elements. Firstly, new characterizations of the
Moore-Penrose inverse of a regular element are given by one-sided
invertibilities. One can find that some parts of the following
Theorem \ref{one-sided MP} were given in \cite[Theorem 3.3]{Zhu
reverse}. Herein, a new proof is given.

\begin{theorem} \label{one-sided MP} Let $a\in R$ be regular with an inner inverse $a^{-}$. Then the following conditions are equivalent{\rm :}

\emph{(i)} $a\in R^\dag$.

\emph{(ii)} $aa^*+1-a a^{-}$ is right invertible.

\emph{(iii)} $a^*a+1-a^{-}a$ is right invertible.

\emph{(iv)} $aa^*a a^{-}+1-a a^{-}$ is right invertible.

\emph{(v)} $a^{-}aa^*a+1- a^{-}a$ is right invertible.

\emph{(vi)} $aa^*+1-aa^{-}$ is left invertible.

\emph{(vii)} $a^*a+1-a^{-}a$ is left invertible.

\emph{(viii)} $aa^*a a^{-}+1-a a^{-}$ is left invertible.

\emph{(ix)} $a^{-}aa^*a+1- a^{-}a$ is left invertible.
\end{theorem}

\begin{proof}

(ii) $\Leftrightarrow$ (iii), (ii) $\Leftrightarrow$ (iv), (iii)
$\Leftrightarrow$ (v), (vi) $\Leftrightarrow$ (vii), (vi)
$\Leftrightarrow$ (viii) and (vii) $\Leftrightarrow$ (ix) follow
from Lemma  \ref{Jlemma}.

(i) $\Rightarrow$ (ii). If $a\in R^\dag$, then there exists $x\in R$
such that $a=aa^*ax$ from Lemma \ref{star regular}. As
$(aa^*aa^{-}+1-aa^{-})(axa^{-}+1-aa^{-})=1$, then
$aa^*aa^{-}+1-aa^{-}$ is right invertible. Hence, $aa^*+1-aa^{-}$ is
right invertible by Lemma \ref{Jlemma}.

(ii) $\Rightarrow$ (i). As $aa^*+1-a a^{-}$ is right invertible,
then $a^*a+1-a^{-}a$ is also right invertible by Lemma \ref{Jlemma}.
Hence, there is $s\in R$ such that $(a^*a+1-a^{-}a)s=1$. We have
$a=a(a^*a+1-a^{-}a)s=aa^*as\in aa^*aR$. So, $a\in R^\dag$ by Lemma
\ref{star regular}.

(i) $\Rightarrow$ (vi). It is similar to the proof of (i)
$\Rightarrow$ (ii).

(vi) $\Rightarrow$ (i). As $aa^*+1-a a^{-}$ is left invertible, then
$t(aa^*+1-a a^{-})=1$ for some $t\in R$. Also, $a=1 \cdot a=
t(aa^*+1-a a^{-})a=taa^*a\in Raa^*a$, which ensures $a\in R^\dag$
according to Lemma \ref{star regular}.
\end{proof}

As a special result of Theorem \ref{one-sided MP}, it follows that

\begin{corollary} \label{MP inverse} {\rm \cite[Theorem 1.2]{Patricio and Mendes}} Let $a\in R$ be regular with an inner inverse $a^{-}$.
Then the following conditions are equivalent{\rm :}

\emph{(i)} $a\in R^\dag$.

\emph{(ii)} $aa^*+1-aa^{-}$ is invertible.

\emph{(iii)} $a^*a+1-a^{-}a$ is invertible.

\emph{(iv)} $aa^*aa^{-}+1-aa^{-}$ is invertible.

\emph{(v)} $a^{-}aa^*a+1-a^{-}a$ is invertible.

\end{corollary}

The following Theorems \ref{right regular and MP} and \ref{left
regular and MP} were given in \cite{Zhu reverse} by authors. Next,
we give different purely ring theoretical proofs.

\begin{theorem} \label{right regular and MP} Let $a\in R$ be regular with an inner inverse $a^{-}$. Then the following conditions are equivalent{\rm :}

\emph{(i)} $a\in R^\dag$ and $aR=a^2R$.

\emph{(ii)} $u=aa^*a+1-aa^{-}$ is right invertible.

\emph{(iii)} $v=a^*a^2+1-a^{-}a$ is right invertible.
\end{theorem}

\begin{proof}

(i) $\Rightarrow$ (ii). As $aR=a^2R$, then $a+1-aa^{-}$ is right
invertible by \cite[Theorem 1]{Puystjens}. Also, from $a\in R^\dag$
we can conclude $aa^*aa^{-}+1-aa^{-}$ is invertible by Corollary
\ref{MP inverse}. Hence,
$u=aa^*a+1-aa^{-}=(aa^*aa^{-}+1-aa^{-})(a+1-aa^{-})$ is right
invertible.

(ii) $\Leftrightarrow$ (iii) follows from Lemma \ref{Jlemma}.

(iii) $\Rightarrow$ (i). Since $v$ is right invertible, there exists
$v_1\in R$ such that $vv_1=1$. Then
$a=avv_1=a(a^*a^2+1-a^{-}a)v_1=aa^*a^2v_1\in aa^*aR$ and hence $a\in
R^\dag$ by Lemma \ref{star regular}. It follows from Corollary
\ref{MP inverse} that $a\in R^\dag$ implies that $w=a^*a+1-a^{-}a\in
R^{-1}$. As $v=(a^*a+1-a^{-}a)(a^{-}a^2+1-a^{-}a)$ is right
invertible, then $a^{-}a^2+1-a^{-}a=w^{-1}v$ is right invertible,
and hence $a+1-a^{-}a$ is also right invertible. So, $aR=a^2R$ by
\cite[Theorem 1]{Puystjens}.
\end{proof}

\begin{remark} \label{ex1} {\rm} {\rm In general, $a\in R^\dag$ and $aR=a^2R$ may not imply $a\in R^\#$. For example,
let $R$ be the ring of all infinite complex matrices with finite
nonzero elements in each column with transposition as involution.
Let $a=\Sigma_{i=1}^\infty e_{i,i+1}\in R$, where $e_{i,j}$ denotes
the infinite matrix whose $(i,j)$-entry is 1 and other entries are
zero. Then $aa^*=1$ and $a^*a=\Sigma_{i=2}^\infty e_{i,i}$. So,
$a^\dag=a^*$ and $aR=a^2R$. But $a\notin R^\#$. In fact, if $a\in
R^\#$, then $a^\#a=aa^\#=aa^\#aa^*=aa^*=1$, which would imply that
$a$ is invertible. This is a contradiction.}
\end{remark}

Dually, we have the following result.

\begin{theorem} \label{left regular and MP} Let $a\in R$ be regular with an inner inverse $a^{-}$. Then the following conditions are equivalent{\rm :}

\emph{(i)} $a\in R^\dag$ and $Ra=Ra^2$.

\emph{(ii)} $u=aa^*a+1-a^{-}a$ is left invertible.

\emph{(iii)} $v=a^2 a^*+1-a a^{-}$ is left invertible.
\end{theorem}

We next give existence criteria and representations of the core
inverse and of the dual core inverse of a regular element in a ring.

\begin{theorem} \label{chen} Let $a\in R$ be regular with an inner inverse $a^{-}$. Then the following conditions are equivalent{\rm :}

\emph{(i)} $a\in R^\# \cap R^\dag$.

\emph{(ii)} $a\in R^{(\#)} \cap R_{(\#)}$.

\emph{(iii)} $u=aa^*a+1-aa^{-}\in R^{-1}$.

\emph{(iv)} $v=aa^*a+1-a^{-}a \in R^{-1}$.

\emph{(v)} $ s=a^*a^2+1-a^{-}a \in R^{-1}$.

\emph{(vi)} $t=a^2a^*+1-aa^{-} \in R^{-1}$.

In this case, \begin{eqnarray*}
&&a^{(\#)}=u^{-1}aa^*,~ a_{(\#)}=a^*av^{-1}, \\
&&a^\dag=(t^{-1}a^2)^*=(a^2s^{-1})^*,\\
&&a^\#=(aa^*t^{-1})^2a=a(s^{-1}a^*a)^2.
\end{eqnarray*}
\end{theorem}

\begin{proof}

(i) $\Leftrightarrow$ (ii) by Lemma \ref{expression}.

(iii) $\Leftrightarrow$ (v) and (iv) $\Leftrightarrow$ (vi) are
obtained by Jacobson's Lemma.

(i) $\Rightarrow$ (iii). From Proposition \ref{group inverse} and
Corollary \ref{MP inverse}, $a\in R^\# \cap R^\dag$ implies that
$a+1-aa^{-}$ and $aa^*aa^{-}+1-aa^{-}$ are both invertible. Hence,
$u=aa^*a+1-aa^{-}=(aa^*aa^{-}+1-aa^{-})(a+1-aa^{-})$ is invertible.

(iii) $\Rightarrow$ (i). Suppose that $u=aa^*a+1-aa^{-}$ is
invertible. Then $a\in R^\dag$ from Theorem \ref{right regular and
MP} and hence $aa^*aa^{-}+1-aa^{-}$ is invertible by Corollary
\ref{MP inverse}. As $u=(aa^*aa^{-}+1-aa^{-})(a+1-aa^{-})$ is
invertible, then $a+1-aa^{-}=(aa^*aa^{-}+1-aa^{-})^{-1}u$ is
invertible, i.e., $a\in R^\#$ by Proposition \ref{group inverse}.

(i) $\Leftrightarrow$ (iv) can be obtained by a similar proof of (i)
$\Leftrightarrow$ (iii).

Next, we give representations of $a^{(\#)}$, $a_{(\#)}$, $a^\dag$
and $a^\#$, respectively.

Since $ua=aa^*a^2$, $a=(u^{-1}aa^*)a^2$. As $a^\#$ exists, then
$a^\#=(u^{-1}aa^*)^2a$ by Lemma \ref{Hartwig group inverse}. From
Lemma \ref{expression}, we have
\begin{eqnarray*}
a^{(\#)}&=&a^\#aa^{(1,3)}=u^{-1}aa^* u^{-1}aa^*a^2a^{(1,3)}\\
&=&u^{-1}aa^*aa^{(1,3)}=u^{-1}aa^*(aa^{(1,3)})^*\\
&=&u^{-1}aa^*.
\end{eqnarray*}

Similarly, it follows that $a^\#=a(a^*av^{-1})^2$ and
$a_{(\#)}=a^*av^{-1}$.

As $as=aa^*a^2$ and $ta=a^2a^*a$, then we have
$a=aa^*(a^2s^{-1})=(t^{-1}a^2)a^*a$. It follows from Lemma
\ref{1314} that $a\in R^\dag$ and
\begin{eqnarray*}
a^\dag&=&(a^2s^{-1})^*a(t^{-1}a^2)^*=(s^{-1})^*(a^2)^*a(a^2)^*(t^{-1})^*\\
&=&(s^{-1})^*(aa^*a^2)^*a^*(t^{-1})^*=(s^{-1})^*(as)^*a^*(t^{-1})^*\\
&=&(a^*)^2(t^{-1})^*\\
&=&(t^{-1}a^2)^*.
\end{eqnarray*}

Similarly, $a^\dag=(a^2s^{-1})^*$.

Noting $sa^{-}a=a^*a^2$, we have $a^{-}a=s^{-1}a^*a^2$ and
$a=aa^{-}a=(as^{-1}a^*)a^2$. Hence, it follows that
$a^\#=(as^{-1}a^*)^2a=a(s^{-1}a^*a)^2$ since $a\in R^\#$.

We can also get $a^\#=(aa^*t^{-1})^2a$ by a similar way.
\end{proof}

\begin{corollary} \label{core inverse 1} Let $a\in R^\dag$. Then the following conditions are equivalent{\rm :}

\emph{(i)} $a\in R^{(\#)}$.

\emph{(ii)} $a\in R_{(\#)}$.

\emph{(iii)} $u=aa^*a+1-aa^{\dag} \in R^{-1}$.

\emph{(iv)} $v=aa^*a+1-a^{\dag} a \in R^{-1}$.

\emph{(v)} $s=a^*a^2+1-a^{\dag} a \in R^{-1}$.

\emph{(vi)} $t=a^2a^*+1-aa^{\dag} \in R^{-1}$.

In this case,
\begin{eqnarray*}
&&  a^{(\#)}=u^{-1}aa^*=aa^*t^{-1},\\
&&a_{(\#)}=a^*av^{-1}=s^{-1}a^*a.
\end{eqnarray*}
\end{corollary}

\begin{proof} As $a\in R^\dag$, then $a\in R^{(\#)}$ if and only if $a\in R^\#$ if and only if $a\in
R_{(\#)}$ by Lemma \ref{expression}. So (i)-(vi) are equivalent by
Theorem \ref{chen}. Moreover, $a^{(\#)}=u^{-1}aa^*$ and
$a_{(\#)}=a^*av^{-1}$. Note that $uaa^*=aa^*t$ and $a^*av=sa^*a$.
Then $u^{-1}aa^*=aa^*t^{-1}$ and $a^*av^{-1}=s^{-1}a^*a$. As
required.
\end{proof}

\begin{proposition} Let $a\in R^\dag$. Then the following conditions
are equivalent{\rm :}

\emph{(i)} $a\in R^{(\#)}$.

\emph{(ii)} $a\in R^\#$.

\emph{(iii)} $a^*+1-aa^\dag\in R^{-1}$.

In this case, $a^\#=(u^{-2})^*a$ and $a^{(\#)}=(u^{-1})^*u^{-1}a^*$,
where $u=a^*+1-aa^\dag$.
\end{proposition}

\begin{proof} (i) $\Leftrightarrow$ (ii) by Theorem \ref{chen} (i)
$\Leftrightarrow$ (ii).

(ii) $\Leftrightarrow$ (iii). Note that
$a^*+1-aa^\dag=(a+1-aa^\dag)^*$. It follows from Proposition
\ref{group inverse} that $a\in R^\#$ if and only if $ a+1-aa^\dag\in
R^{-1}$ if and only if $a^*+1-aa^\dag\in R^{-1}$.

Let $u=a^*+1-aa^\dag$. Then $u^*a=a^2$ and $a=(u^*)^{-1}a^2$. As
$a\in R^\#$, then $a^\#=(u^*)^{-2}a=(u^{-2})^*a$ by Lemma
\ref{Hartwig group inverse}.

Since $a\in R^\dag$, it follows that
\begin{eqnarray*}
 a^{(\#)}&=&a^\#aa^{(1,3)}=a^\#aa^\dag=(u^*)^{-2}a^2a^\dag\\
 &=&(u^*)^{-1} (u^*)^{-1}a^2a^\dag=(u^*)^{-1}aa^\dag\\
 &=&(u^*)^{-1}u^{-1}uaa^\dag\\
 &=&(u^{-1})^*u^{-1}a^*.
 \end{eqnarray*}
The proof is completed.
\end{proof}

\begin{proposition}
Let $a\in R^\#$. Then the following conditions are equivalent{\rm :}

\emph{(i)} $a\in R^{(\#)} \cap R_{(\#)}$.

\emph{(ii)} $a\in R^\dag$.

\emph{(iii)} $a^*+1-aa^\#\in R^{-1}$.

In this case, $a^\dag=(u^{-1})^*a(u^{-1})^*$,
$a^{(\#)}=a^\#a(u^{-1})^*$ and $a_{(\#)}=(u^{-1})^*aa^\#$, where
$u=a^*+1-aa^\#$.
\end{proposition}

\begin{proof}
(i) $\Leftrightarrow$ (ii) by Theorem \ref{chen} (i)
$\Leftrightarrow$ (ii).

(ii) $\Rightarrow$ (iii). Note that $a\in R^\dag$ implies
$a^*a+1-a^\#a \in R^{-1}$ by Corollary \ref{MP inverse}. As $a\in
R^\#$, then $a+1-aa^\dag \in R^{-1}$ from Proposition \ref{group
inverse}. Since $a^*a+1-a^\#a=(a^*+1-aa^\#)(a+1-aa^\dag)$, it
follows that $a^*+1-aa^\#=(a^*a+1-a^\#a)(a+1-aa^\dag)^{-1}\in
R^{-1}$.

(iii) $\Rightarrow$ (ii). Let $u=a^*+1-aa^\#$. Then $ua=a^*a$ and
$au=aa^*$. As $u\in R^{-1}$, then $a=aa^*u^{-1}=u^{-1}a^*a\in aa^* R
\cap Ra^*a$. So, $a\in R^\dag$ and $(u^{-1})^*$ is both a
\{1,3\}-inverse and a \{1,4\}-inverse of $a$. Moreover,
$a^\dag=a^{(1,4)}aa^{(1,3)}=(u^{-1})^*a(u^{-1})^*$.

Hence, $a^{(\#)}=a^\#aa^{(1,3)}=a^\#a(u^{-1})^*$ and
$a_{(\#)}=a^{(1,4)}aa^\#=(u^{-1})^*aa^\#$.
\end{proof}

It is known that if $a\in R^\dag$ then $aa^{(1,3)}=aa^\dag$.
Applying Corollary \ref{add new}, it follows that

\begin{corollary} Let $R$ be a Dedekind-finite ring. If $a\in R^\dag$, then $a\in R^{(\#)}$ if and only if $a^*a+1-aa^\dag \in R^{-1}$.
In this case, $a^{(\#)}=(a^*a+1-aa^\dag)^{-1}a^*$.
\end{corollary}

\begin{remark}
{\rm  Suppose $2\in R^{-1}$. If $a^*a+1-aa^\dag \in R^{-1}$ implies
$a\in R^{(\#)}$ for any $a\in R^\dag$, then $a^*a=1$ can conclude
$aa^*=1$. Indeed, if $a^*a=1$, then $a\in R^\dag$ and $a^\dag=a^*$.
Hence, $a^*a+1-aa^\dag=2-aa^\dag \in R^{-1}$ with inverse
$\frac{1}{2}(1+aa^\dag)$. Thus, $a\in R^{(\#)}$ and $a\in R^\#$. As
$aa^\#=a^\#a=(a^*a)a^\#a=a^*a=1$, then $a\in R^{-1}$ and hence
$aa^*=1$.}
\end{remark}

\centerline {\bf ACKNOWLEDGMENTS} The authors are highly grateful to
the referee for his/her valuable comments which led to improvements
of this paper. This research was carried out by the first author
during his visit to the Department of Mathematics and Applications,
University of Minho, Portugal. The first and second authors
gratefully acknowledge the financial support of China Scholarship
Council. This research is also supported by the National Natural
Science Foundation of China (No. 11371089), the Natural Science
Foundation of Jiangsu Province (No. BK20141327), the Scientific
Innovation Research  of College Graduates in Jiangsu Province (No.
CXLX13-072), the Scientific Research Foundation of Graduate School
of Southeast University, the Portuguese Funds through FCT-
`Funda\c{c}\~{a}o para a Ci\^{e}ncia e a Tecnologia', within the
project UID-MAT-00013/2013.

\end{document}